\documentclass[11pt,reqno]{amsart} 

\usepackage{amssymb,latexsym}
\usepackage[draft=true]{hyperref}
\usepackage{hyperref}
\usepackage{cite} 

\usepackage[height=190mm,width=130mm]{geometry} 

\theoremstyle{plain}

\newtheorem{proposition}{Proposition}

\theoremstyle{definition}
\newtheorem{definition}{Definition}
\newtheorem{Particular Cases}{Particular Cases}
\newtheorem{Particular Case}{Particular Case}

\theoremstyle{remark}

\newtheorem{example}{Example}

\numberwithin{equation}{section} 

\begin{document}
\title[Ricci-Yamabe maps for Riemannian flows]{Ricci-Yamabe maps for Riemannian flows and their volume variation and volume entropy} 

\author{Mircea Crasmareanu}
\address{Faculty of Mathematics \\ University "Al. I.Cuza" \\ Iasi, 700506 \\ Rom\^{a}nia\newline
 http://www.math.uaic.ro/$\sim$mcrasm}

\email{mcrasm@uaic.ro}

\thanks{The second author was supported  by  The Scientific and Technological Research Council of Turkey (T\"UB{\.I}TAK)  Grant No.1059B141600696.}

\author{S{\.I}nem G\"uler}

\address{Department of Mathematics \\
Ayaza\u ga Campus, Faculty of Science and Letters, \\
Istanbul Technical University \\
34469 Maslak, Istanbul \\
Turkey}


\email{singuler@itu.edu.tr}

\begin{abstract}
The aim of this short note is to produce new examples of geometrical flows associated to a given Riemannian flow $g(t)$. The considered flow in covariant symmetric $2$-tensor fields will
be called Ricci-Yamabe map since it involves a scalar combination of Ricci tensor and scalar curvature of $g(t)$. Due to the signs of considered scalars the Ricci-Yamabe flow can be also
a Riemannian or semi-Riemannian or singular Riemannian flow. We study the associated function of volume variation as well as the volume entropy. Finally, since the two-dimensional case
was the most handled situation we express the Ricci flow equation in all four orthogonal separable coordinate systems of the plane.
\end{abstract}




\maketitle

\section*{Introduction}

The recent applications of the Hamilton-Ricci flow (\cite{c:ln}) in Perelman's proof of the geometrization conjecture and Brendle-Schoen's proof of the differentiable sphere theorem attract the attention to the geometry of Riemannian flows. These flows and other general geometric flows have been applied to several topological, geometrical and physical problems and we cite only some of them: the mean curvature flow, the K\"ahler-Ricci, Calabi and Yamabe flows, the curve shortening flows and so on.

The present paper aims to introduce a scalar combination of Ricci and Yamabe flow under the name of {\it Ricci-Yamabe map}. Due to the signs of involved scalars ($\alpha $ and $\beta $)
the Ricci-Yamabe map can be also a Riemannian or semi-Riemannian or singular Riemannian flow. This kind of multiple choices can be useful in some geometrical or physical (e.g.
relativistic) theories. Another strong motivation for our map is the fact that, although the Ricci and Yamabe flow are identical in dimension two, they are essentially different in
higher dimensions. Let us remark that an interpolation flow between Ricci and Yamabe flows is considered in \cite{c:c} under the name {\it Ricci-Bourguignon flow} but it depends on a
single scalar, see also the pages 79-80 of the book \cite{c:g}. A normalized version of Ricci-Bourguignon-Yamabe flow is studied in \cite{m:t} from the point of view of spectral
geometry.

In the first section we introduce and discuss five examples of Ricci-Yamabe maps: conformal (particular cone), convex-Euclidean, generalized Poincar\'e, generalized cigar flow and time dependent 2D warped metrics; a sixth example, consisting in gradient flows introduced by Min-Oo and Ruh in \cite{mo:r}, is very briefly recalled due to its relationship with the Ricci flow. For some of them the general dimension $n$ is too complicated from the point of view of their Ricci tensor and scalar curvature and then the case $n=2$ is particularly analyzed with respect to the (global sometimes) coordinates $(u, v)$; on this way the present study can be useful in the shape analysis of surfaces following the path of \cite{z:g}. Also, for an $(\alpha , \beta )$-Ricci-Yamabe flow we express the variation of some geometrical quantities: the Christoffel symbols, the scalar curvature and the volume form.

Another main topic studied for the introduced Ricci-Yamabe map is the volume variation to which the second section is devoted. The given Riemannian flow $g(t)$ will be called $(\alpha , \beta)$-{\it RY-expanding} or {\it steady} or ${\it shrinking}$ if this variation is positive, zero or negative, respectively. For the example of generalized cigar flow with an exponential potential we derive all three possibilities: expanding, steady and shrinking. Also, we get two examples of $(\alpha , \beta)$-RY-expanding flows which are uniform i.e. they depend only on the time variable $t$ and not on the coordinates of $M$. We finish the second section with a study of the volume entropy for an $(\alpha , \beta)$-Ricci-Yamabe flow following the technique of \cite{v:c}.

In the last section we study a generalization of the conformal example. More precisely, we allow a general conformal transformation $f(t, x, y)(dx^2+dy^2)$ of the Euclidean plane metric
and its corresponding Ricci flow equation. The previous coordinate system $(x, y)$ being provided by the Euclidean geometry can be without a geometrical or physical significance for a
different geometry. But it has been known that for some constant curvature spaces there exist {\it orthogonal separable coordinate systems}, namely coordinate systems for which a given
Hamiltonian system in classical me\-cha\-nics and Schr\"odinger equation of the quantum mechanics admit solutions via separation of variables; hence these coordinate systems are related
to the superintegrability problem, \cite{b:cp}. For example, in the real 2D space there are four such systems, for the complex plane there are six while for 2D sphere there are only two.
We derive the expression of the Ricci flow equation in all 2D coordinates systems and also for their solitonic analogs. On this way, we add to the PDE flavor of the topic of Ricci flow
three new equations: the polar, the parabolic and the elliptic Ricci flow equation.


\section{The Ricci-Yamabe map of a Riemannian flow}

For a smooth $n$-dimensional manifold $M^n$ let $T^s_2(M)$ be the linear space of its symmetric tensor fields of $(0, 2)$-type and $Riem(M)\subsetneqq T^s_2(M)$ be the infinite space of its Riemannian metrics. Following the general theory of geometric flows we introduce:

\begin{definition}\label{def:1.1}
A {\it Riemannian flow} on $M$ is a smooth map:
\begin{equation}
\label{1.1}
g:I\subseteq \mathbb{R}\rightarrow Riem(M)
\end{equation}
where $I$ is a given open interval. We can call it also as {\it time-dependent} (or {\it non-stationary}) {\it Riemannian metric}.
\end{definition}

Throughout this work we fix a Riemannian flow $g(\cdot )$ and one denotes by $Ric(t)$ the Ricci tensor field of $g(t)$ and by $R(t)$ the corresponding scalar curvature. Let also be given $\alpha $ and $\beta $ some scalars.

\begin{definition}\label{def:1.2}
 The map $RY^{(\alpha , \beta , g)}:I\rightarrow T^s_2(M)$ given by:
\begin{equation}
\label{1.2}
RY^{(\alpha, \beta ,g)}(t):=\frac{\partial g}{\partial t}(t)+2\alpha Ric(t)+\beta R(t)g(t)
\end{equation} 
is called {\it the $(\alpha , \beta)$-Ricci-Yamabe map} of the Riemannian flow $(M, g)$. If $RY^{(\alpha, \beta ,g)}\equiv 0$ then $g(\cdot )$ will be called {\it an $(\alpha , \beta)$-Ricci-Yamabe flow}.

\end{definition}

\begin{Particular Case}\label{part case:1.3}
We have:
\begin{enumerate}
\item[(1)]  the equation $RY^{(1, 0, g)}\equiv 0$ is the Ricci flow while the equation $RY^{(0, 1, g)}\equiv 0$ is the Yamabe flow
(\cite[p. 520]{c:ln}). The equation of an $(\alpha /2, \beta )$-Ricci-Yamabe flow appears at page 50 in \cite{b:u} where the Ricci curvature is interpreted as a vector field on the space of metrics.   
\item[(2)] for $n=2$ since $Ric=\frac{R}{2}g=Kg$ with $K$ being the Gaussian curvature we get:
\begin{equation}
\label{1.3}
RY^{(\alpha , \beta , g)}(t)=\frac{\partial g}{\partial t}(t)+2(\alpha +\beta )K(t)g(t).
\end{equation}
Hence the Ricci and Yamabe flows coincide on surfaces. 
\item[(3)] Due to the signs of $\alpha $ and $\beta $ the Ricci-Yamabe map can be also a Riemannian or semi-Riemannian or singular (degenerate) Riemannian flow and this kind of freedom can be useful in some geometrical or physical (e.g. relativistic) theories. For example, a recent bi-metric approach of spacetime geometry appears in \cite{a:ks} and \cite{b:c}.
\end{enumerate}

\end{Particular Case}

\begin{example}\label{exp:1.4.}
 
$(1)$ (\textit{Cone and conformal flow}) Fix $g\in Riem(M)$ and a smooth $f:I\rightarrow \mathbb{R}^*_{+}=(0, +\infty )$. We define $g(t):=f(t)g$ which we call {\it conformal flow}. In particular, for $I=(0, +\infty )$ and $f(t)=t$ we obtain {\it the cone flow}. Applying Exercise 1.11 of \cite[p. 6]{c:ln} one have:
\begin{equation}
\label{1.4}
Ric(t)=Ric(g), \quad R(t)=\frac{1}{f(t)}R(g) 
\end{equation}
and then the $(\alpha , \beta )$-{\it Ricci-Yamabe-conformal map} is:
\begin{equation}
\label{1.5}
RY^{(\alpha , \beta , g)}(t)=[f^{\prime }(t)+\beta R(g)]g+2\alpha Ric(g),
\end{equation}
which becomes constant for the cone flow. In particular, if $g$ is an Einstein metric, from:
\begin{equation}
\label{1.6}
Ric(g)=\frac{R(g)}{n}g,
\end{equation}
we get:
\begin{equation}
\label{1.7}
RY^{(\alpha , \beta , g)}(t)=[f^{\prime }(t)+(\beta +\frac{2\alpha }{n})R(g)]g,
\end{equation}
which is a time-dependent Einstein metric. \\

$(2)$ A fixed Riemannian geometry $(M, g)$ is called {\it convex-Euclidean} (\cite{b:cp}) if it supports a Riemannian flow:
\begin{equation}
\label{1.8}
g(t)=(1-t)g+tI, 
\end{equation}
with $t\in [0, 1]$ and $I$ the covariant version of the Kronecker tensor field. For example, every pa\-ral\-le\-li\-za\-ble manifold, in particular any Lie group, is a convex-Euclidean one. In the following, due to complicated equations on the general case we restrict to $n=2$ and suppose that $g$ is isothermal: $g(u, v):=E(u, v)I$ with the smooth function $E>0$ on the surface $M$. Then all terms of the convex-Euclidean flow are isothermal:
\begin{equation}
\label{1.9}
g(t)=[(1-t)E+t]I.
\end{equation}
The expression of the Gaussian curvature for isothermal metrics is well-known and hence:
\begin{equation}
\label{1.10}
K(t)=-\frac{1}{2[(1-t)E+t]}\Delta _{u, v}(\ln [(1-t)E+t]),
\end{equation}
where $\Delta _{u, v}$ is the usual 2D Laplacian: $\Delta _{u, v}=\partial ^2_{uu}+\partial ^2_{vv}$. In conclusion:
\begin{equation}
\label{1.11}
RY^{(\alpha , \beta , g)}(t)=[1-E-(\alpha +\beta )\Delta _{u, v}(\ln [(1-t)E+t])]I,
\end{equation}
which is a time-dependent isothermal (semi-, singular) metric. \\

$(3)$ (\textit{Poincar\'e flow}) For the Poincar\'e half space model of hyperbolic geometry in $\mathbb{R}^n_{+}=\{(x^1,...,x^n)\in \mathbb{R}^n; x^n>0\}$ (\cite[p. 135]{p:p}) we define {\it the Poincar\'e flow} on $I=\mathbb{R}$ as:
\begin{equation}
\label{1.12}
g(t)=\frac{1}{(x^n)^t}\left[d(x^1)^2+...+d(x^n)^2\right]=\frac{1}{(x^n)^t}g^n_e,
\end{equation}
which satisfies $\partial _tg(t)=(-\ln x^n)\cdot g(t)$ and can be considered as a conformal flow for the Euclidean metric $g^n_e$ on $\mathbb{R}^n_{+}$. Applying the formulae of \cite[p. 35]{c:ln} we have:
\begin{equation}
\label{1.13}
Ric(t)=\frac{(2-n)t^2-2t}{4(x^n)^2}g_e^{n-1}+\frac{(1-n)t}{2(x^n)^2}d(x^n)^2, \quad R(t)=\frac{1-n}{(x^n)^{2-t}}\left[t+\frac{(n-2)t^2}{4}\right]
\end{equation}
and then:
\begin{align}
\label{1.14}
RY^{(\alpha , \beta , g)}(t)\cdot (x^n)^2=&\{(1-n)\beta \left[t+\frac{(n-2)t^2}{4}\right]-(x^n)^{2-t}\ln x^n\}g^n_e\\ \notag
+&[(1-\frac{n}{2})t^2-t]\alpha g_e^{n-1}+(1-n)\alpha td(x^n)^2. 
\end{align}
For $n=2$ we obtain:
\begin{equation}
\label{1.15}
RY^{(\alpha , \beta , g)}(t)=\frac{-(x^2)^{2-t}\ln x^2-(\alpha +\beta )t}{(x^2)^2}g^2_e. 
\end{equation}

$(4)$ (\textit{Generalized cigar flow}) In \cite[p. 154]{c:ln} it is given the Hamilton cigar $2D$ metric on $M=\mathbb{R}^2$ as a steady Ricci soliton:
\begin{equation}
\label{1.16}
g_c(u, v)=\frac{1}{1+u^2+v^2}I
\end{equation}
and then we introduce {\it the generalized cigar flow}:
\begin{equation}
\label{1.17}
g^f_c(t)=\frac{1}{f(t)+u^2+v^2}I
\end{equation}
for a smooth $f:\mathbb{R}\rightarrow \mathbb{R}^{*}_+$ with $f(0)=1$. Its Ricci-Yamabe map is:
\begin{equation}
\label{1.18}
RY^{(\alpha , \beta , g^f)}(t)=\frac{4(\alpha +\beta )f(t)-f^{\prime }(t)}{(f+u^2+v^2)^2}I
\end{equation}
and then $RY^{(\alpha , \beta , g^f)}\equiv 0$ if and only if $f(t)=e^{4(\alpha +\beta )t}$. Recalling that for $\alpha =1$ and $\beta =0$ we have a Ricci flow we re-obtain the result of \cite{c:ln} that $(1.17)$ is a Ricci flow for $f(t)=e^{4t}$. \\

$(5)$ (\textit{Time dependent warped metrics}) Let the 2D Riemannian flow:
\begin{equation}
\label{1.19}
g(t)=du^2+f(t)G(u)dv^2
\end{equation}
again for a smooth $f:\mathbb{R}\rightarrow \mathbb{R}^{*}_+$ with $f(0)=1$. Its $(\alpha , \beta )$-Ricci-Yamabe map is:
\begin{equation}
\label{1.20}
RY^{(\alpha , \beta , g)}(t)=f^{\prime }(t)G(u)dv^2-\frac{\alpha +\beta }{\sqrt{f(t)G}}\Delta \ln G\cdot g(t).
\end{equation}
For $G(u)=u^2$ and $f(t)=t^2$ which do not satisfies $f(0)=1$ the metric $\eqref{1.19}$ appears in the Exercise 1.6.7 of \cite[p. 33]{p:p}. With $G(u)=sn_k^2(u)$ of \cite[p. 117]{p:p} we obtain the time dependent rotationally symmetric metrics generalizing the constant curvature ($=k$) metrics. Hence this flow $\eqref{1.19}$ has the $(\alpha , \beta )$-Ricci-Yamabe map:
\begin{equation}
\label{1.21}
RY^{(\alpha , \beta , g)}(t)=f^{\prime }(t)sn_k^2(u)dv^2-\frac{2k(\alpha +\beta )}{\sqrt{f(t)}}\cdot g(t).
\end{equation}
Let us note that Ricci flows on general warped product metrics are studied in \cite{wj:l}. \\

$(6)$ (\textit{Gauge flows of Min-Oo-Ruh type}) Starting with a given Riemannian metric $g$, a gradient flow is considered in \cite{mo:r} by means of a family $\theta _t$ of tensor fields of $(1, 1)$-type:
\begin{equation}
\label{1.22}
g(t)=g(\theta _t \cdot , \theta _t\cdot )
\end{equation}
with $\theta _0=I$. The infinitesimal gauge transformation $\dot{\theta }:=\frac{d}{dt}\theta _t|_{t=0}$ is used to express the infinitesimal change of the Levi-Civita connection. 

The
choice $\dot{\theta }=-Ric$ of $(1, 1)$-type yields exactly the Hamilton's Ricci flow and the Lagrangian of the gradient flow then becomes a Yang-Mills functional $D\rightarrow \int
_M\|F^D\|^2$ on the space of Cartan connections of hyperbolic type. 

The choice $\dot{\theta }=-\alpha Ric-\frac{\beta }{2}RI$ yields the $(\alpha , \beta )$-Ricci-Yamabe flow. Moreover,
given the vector field $\xi $ and $\lambda \in \mathbb{R}$ the choice $Ric=-\nabla ^g\xi -\lambda I$ with $\nabla ^g\xi $ the covariant derivative of $\xi $ with respect to the fixed
metric $g$ yields {\it the Ricci soliton} $(g, \xi , \lambda )$ on $M$ i.e. a self-similar solution of the Ricci flow, \cite{c:ln}. \quad $\Box $

\end{example} 

We finish this section with the variation of some geometrical objects along an $(\alpha , \beta)$-Ricci-Yamabe flow:

\begin{proposition}\label{prop:1.5}
Let $g(t)$ be an $(\alpha , \beta)$-Ricci-Yamabe flow. Then: 
\begin{itemize}
\item[(1)] {\it the variation of the Christoffel symbols is}:
\begin{align}
\label{1.23}
\partial _t\Gamma ^k_{ij}(t)=&\alpha g^{kl}\left(\nabla _lR_{ij}-\nabla _iR_{jl}-\nabla_jR_{il}\right)(t)\\ \notag
-&\frac{\beta }{2}\left[(\nabla _iR)(t)\delta ^k_j+(\nabla _jR)(t)\delta ^k_i-(\nabla ^kR)(t)g_{ij}(t)\right], 
\end{align}
\item[(2)] {\it the variation of the scalar curvature is}:
\begin{equation}
\label{1.24}
\partial _tR(t)=[\alpha +(n-1)\beta ]\Delta _{g(t)}R(t)+2\alpha \|Ric(t)\|_{g(t)}^2+\beta R^2(t), 
\end{equation}
{\it with $\Delta _{g(t)}$ the Laplacian with respect to $g(t)$; for $n=2$ this relation becomes}:
\begin{equation}
\label{1.25}
\left\{ \begin{array}{ll}
\partial _tR(t)=(\alpha +\beta )\left[\Delta _{g(t)}R(t)+R^2(t)\right], \\
\partial _tK(t)=(\alpha +\beta )\left[\Delta _{g(t)}K(t)+2K^2(t)\right],\\
\end{array} \right.
\end{equation}
\item[(3)] the variation of the volume form is:
\begin{equation}
\label{1.26}
\partial _td\mu (t)=-(\alpha +\frac{n}{2}\beta )R(t)d\mu (t).
\end{equation}
 It follows that if $2\alpha +n\beta =0$, then the volume is constant and the variation of the scalar curvature is:
\begin{equation}
\label{1.27}
\frac{1}{\alpha}\partial _tR(t)=\frac{2-n}{n}\Delta _{g(t)}R(t)+2\|Ric(t)\|_{g(t)}^2-\frac{2}{n}R^2(t), 
\end{equation}
provided $\alpha \neq 0$; particularly is zero in dimension $n=2$. For an $(\alpha , \beta)$-Ricci-Yamabe flow on a closed Riemannian surface $M^2$ with $\alpha +\beta \geq 1$ and $t\in (0, T]$, we have the next global inequality on $M$:
\begin{equation}
\label{1.28}
K(t)\geq -\frac{1}{2(\alpha +\beta )t}. 
\end{equation}
\end{itemize}
\end{proposition}

\begin{proof}
We apply the formulae of \cite{c:ln} with the variation tensor field:
\begin{equation}
\label{1.29}
v(t)=-2\alpha Ric(t)-\beta R(t)g(t)
\end{equation}
and the corresponding variation function:
\begin{equation}
\label{1.30}
V(t)=-(2\alpha +n\beta )R(t). 
\end{equation}
We have $div\ v=-(\alpha +\beta )dR$ and $div(div\ v)=-(\alpha +\beta )\Delta R$. For the item (3) at the Ricci flow $\alpha =1$ we derive the corresponding {\it normalized Ricci flow} of \cite[p. 128]{c:ln} with the associated scalar $\beta =-\frac{2}{n}$ multiplied with the average scalar curvature. The inequality $(1.28)$ follows from the maximum principle. 
\end{proof}

\begin{Particular Case}\label{particular case:1.6}
 Suppose that the given Riemannian flow $g(t)$ is {\it Ricci-recurrent} i.e. for a fixed time-dependent $1$-form $\eta $ we have:
\begin{equation}
\label{1.31}
\nabla _l(t)R_{ij}(t)=\eta _l(t)R_{ij}(t). 
\end{equation}
Then $\eqref{1.23}$ becomes, with $\eta ^k=g^{ka}\eta _a$:
\begin{align}
\label{1.32}
\partial _t\Gamma ^k_{ij}(t)=&\alpha \left(\eta ^k(t)R_{ij}(t)-\eta _i(t)R_{j}^k(t)-\eta _j(t)R_{i}^k(t)\right)\\ \notag
-&\frac{\beta }{2}\left(\eta  _i(t)\delta ^k_j+\eta _j(t)\delta ^k_i-g_{ij}\eta ^k(t)\right)R(t),
\end{align}
since $\nabla _l(t)R(t)=\eta _lR(t)$. The relation $\eqref{1.25}$ becomes:
\begin{equation}
\label{1.33}
\partial _tR(t)=(\alpha +\beta )\left[(div _{g(t)}\eta (t)+\|\eta (t)\|^2_{g(t)})R(t)+R^2(t)\right].
\end{equation}
In fact, in dimension $2$ if $K(t)>0$, then $\eqref{1.31}$ holds for:
\begin{equation}
\label{1.34}
\eta =d(\ln K)
\end{equation}
and hence $\eqref{1.32}$ reads:
\begin{equation}
\label{1.35}
\partial _tK(t)=(\alpha +\beta )\left[(\Delta _{g(t)}\ln K(t)+\|\nabla \ln K(t)\|^2_{g(t)})K(t)+2K^2(t)\right].
\end{equation}
 
\end{Particular Case}

\section{The volume variation of Ricci-Yamabe maps and volume entropy of a Ricci-Yamabe flow}

Inspired by \cite{m:c} and \cite{c:k} we introduce:

\begin{definition}\label{def:2.1}
 The {\it $(\alpha , \beta )$-Ricci-Yamabe volume variation} for the Riemannian flow $g(\cdot )$ is the smooth function $V^{(\alpha , \beta , g)}:M\times I\rightarrow M\times I$ given by:
\begin{equation}
\label{2.1}
\partial _tV^{(\alpha , \beta , g)}(t):=Tr _{g(t)}RY^{(\alpha , \beta , g)}(t)=\sum _{i, j=1}^ng^{ij}(t){RY}^{(\alpha , \beta , g)}_{ij}(t)
\end{equation}
where, as usual, $g^{ij}(t)$ are the components of inverse $g^{-1}(t)$. The flow will be called $(\alpha , \beta)$-{\it RY-expanding} or {\it steady} or ${\it shrinking}$ if this variation is positive, zero or negative, respectively. Moreover, if this variation depends only on $t$, then we say that the flow is {\it uniform}.
\end{definition}

\begin{example}\label{exp:2.2}

$(1)$ For conformal flows we have from $\eqref{1.5}$:
\begin{equation}
\label{2.2}
\partial _tV^{(\alpha , \beta , g)}(t)=\frac{nf^{\prime }(t)+(2\alpha +n\beta )R(g)}{f(t)}
\end{equation}
which yields:
\begin{equation}
\label{2.3}
V(t)^{(\alpha , \beta , g)}=n\ln f(t)+(2\alpha +n\beta )R(g)F(t)
\end{equation}
with $F$ an anti-derivative for $\frac{1}{f}$. In particular, for $2\alpha +n\beta =0$ or flat metric $g$, the conformal flow is uniform $(\alpha , \beta )-RY-$expanding. \\

$(2)$ For the two-dimensional Poincar\'e flow by using $\eqref{1.15}$ we get:
\begin{equation}
\label{2.4}
V(t)^{(\alpha , \beta , g)}=-\frac{\alpha +\beta}{2}t^2+(x^2)^{2-t}.
\end{equation}
Hence, on the horizontal lines $x^2=constant$ this flow is $(\alpha , \beta )$-RY-expanding, steady or shrinking according to $\alpha +\beta $ being a negative, null or positive number. The existence of a preferred direction in the physical space is discussed in \cite[p. 50]{b:l}. \\

$(3)$ For the generalized cigar flow its variation of volume is:
\begin{equation}
\label{2.5}
\partial _tV^{(\alpha , \beta , g)}(t)=\frac{4(\alpha +\beta )f-f^{\prime }}{f+u^2+v^2}.
\end{equation}
For example, choosing an exponential potential $f(t)=e^{ct}$ we derive that the flow is $(\alpha , \beta)$-RY-expanding, steady or shrinking according to $\alpha +\beta $ being greater, equal or lower than $\frac{c}{4}$.  \\

$(4)$ For the time dependent warped metrics $\eqref{1.19}$ with $G(u)=sn_k^2(u)$ from $\eqref{1.21}$ we have:
\begin{equation}
\label{2.6}
V^{(\alpha , \beta , g)}(t)=\ln f(t)-4k(\alpha +\beta )F(t)
\end{equation}
with $F$ and anti-derivative for $\frac{1}{\sqrt{f}}$. In particular, for $\alpha +\beta =0$ or flat metric $g$, this flow is uniform $(\alpha , \beta )-RY-$expanding. \quad $\Box $

\end{example}

More generally, the steady case is characterized by:

\begin{proposition}\label{2.3}
 The Riemannian flow $g(\cdot )$ is $(\alpha , \beta)$-RY-steady if and only if its scalar curvature is:
\begin{equation}
\label{2.7}
(2\alpha +n\beta )R(t)=-Tr_{g(t)}(\partial _tg(t)).
\end{equation}
\end{proposition}

\begin{proof}
From $\eqref{2.1}$ we derive that the steady case is characterized by:
\begin{equation}
\label{2.8}
n\beta R(t)=-Tr_{g(t)}(\partial _tg(t)+2\alpha Ric(t)).
\end{equation}
Since $Tr$ is a linear operator we derive immediately the conclusion. 
\end{proof}

We finish this section with a study of the volume entropy for an $(\alpha , \beta)$-Ricci-Yamabe flow supposing that $M$ is compact and with $n\geq 3$. Let $\pi : \tilde{M}\rightarrow M$ be the universal cover of $M$ and $\omega (M, g, r>0)=\frac{1}{r}\ln V_{\tilde{g}}(\tilde{B}(r))$ be {\it the volume growth function} of the Riemannian metric $g$; here $V_{\tilde{g}}$ is the volume with respect to the cover metric $\tilde{g}$. Then {\it the volume entropy} of $(M, g)$ is: $h(M, g)=\lim _{r\rightarrow \infty }\omega (M, g, r)$.

We have a result similar to Proposition 2.1 of \cite{v:c}:

\begin{proposition}\label{prop:2.4}
 Let $g(t)$ be an $(\alpha , \beta)$-Ricci-Yamabe flow on the compact $M$ satisfying: 
 \begin{itemize}
\item[(i)] {\it the injectivity radius $i(M, g(t))\geq i\in \mathbb{R}$ uniformly in time}; {\it in particular $g(t)$ has the sectional curvature uniformly bounded from above by a positive constant and
an uniform lower bound of the lengths of closed geodesics},
\item[(ii)] $n\geq 3$ {\it and} $2\alpha +n\beta >0$. 
\end{itemize}
 Then the volume entropy $h(M, g(t))$ is nondecreasing.
\end{proposition}

\begin{proof}

 With a computation similar to that of \cite{v:c} and using $\eqref{1.26}$ we obtain:
\begin{equation}
\label{2.9}
\partial _t \omega (M, g(t), r)=-\frac{2\alpha +n\beta }{2rV_{\tilde{g}(t)}(\tilde{B}(r))}\int _{\tilde{B}(r)}\tilde{R}(t)d\tilde{ \mu }(t).
\end{equation}
The hypothesis (i) assures the existence of a constant $C>0$ such that:
\begin{equation}
\label{2.10}
\lim _{r\rightarrow \infty}\frac{1}{V_{\tilde{g}(t)}(\tilde{B}(r))}\int _{\tilde{B}(r)}\tilde{R}(t)d\tilde{ \mu }(t)\leq n(n-1)C.
\end{equation}
Hence, with (ii) we get $\partial _th(M, g(t))\geq 0$ which means the conclusion. 

\end{proof}

\section{The orthogonal companions of the 2D Ricci flow equation and the associated solitonic PDEs}

In this section we start with the general conformal-Euclidean 2D flow on an open subset of the plane, \cite[p. 286]{ci:co}:
\begin{equation}
\label{3.1}
g(t)=\exp(h(x, y, t))[dx^2+dy^2]
\end{equation}
for which the Ricci flow equation is:
\begin{equation}
\label{3.2}
(\exp h)_t=\Delta h 
\end{equation}
with $\Delta $ the Euclidean Laplacian: $\Delta h=h_{xx}+h_{yy}$.

In the following, we consider the Ricci flow equation $\eqref{3.2}$ in the separable coordinate systems of $\mathbb{R}^2$ having the  model as the paper \cite{b:cp}. \\

$(1)$ The Cartesian coordinates $(x, y)$ form the first separable coordinate system of the plane. With the separation of variables $h=f(t, x)g(t, y)$ we have:
\begin{equation}
\label{3.3}
(f_tg+fg_t)\exp (fg)=f_{xx}g+fg_{yy}, 
\end{equation}
while the separation of variables $h=f(t, x)+g(t, y)$ yields the equation:
\begin{equation}
\label{3.4}
(f_t+g_t)\exp (f+g)=f_{xx}+g_{yy}. 
\end{equation}

The following three systems are: \\

$(2)$  Polar coordinates: $u=x\cos y$, $v=x\sin y$. The equation $\eqref{3.2}$ becomes {\it the polar Ricci flow equation}:
\begin{align}
\label{3.5}
h_{uu}(\cos ^2y+x^2\sin ^2y)+&h_{vv}(\sin ^2y+x^2\cos ^2y)\\ \notag
+&2h_{uv}\sin y\cos y(1-x^2)=(\exp h)_t. 
\end{align}
\begin{itemize}
\item[(2.1)] By searching for $h$ of the form $h=\varphi (t, w=u+\alpha v)$ we get the {\it polar solitonic Ricci flow equation}:
\begin{align}
\label{3.6}
\varphi _{ww}[\cos ^2y+&x^2\sin ^2y+\alpha ^2(\sin ^2y+x^2\cos ^2y)\\ \notag
+&2\alpha \sin y\cos y(1-x^2)]=\varphi _t\exp \varphi . 
\end{align}
\item[(2.2)] By searching for $h$ of the form $h=f(t, u)g(t, v)$ we obtain:
\begin{align}
\label{3.7}
f_{uu}g(\cos ^2y+x^2\sin ^2y)+&fg_{vv}(\sin ^2y+x^2\cos ^2y)\\ \notag
+&2f_{u}g_{v}\sin y\cos y(1-x^2)=(f_tg+fg_t)\exp(fg). 
\end{align}
\item[(2.3)] By searching for $h$ of the form $h=f(t, u)+g(t, v)$ we derive:
\begin{equation}
\label{3.8}
f_{uu}(\cos ^2y+x^2\sin ^2y)+g_{vv}(\sin ^2y+x^2\cos ^2y)=(f_t+g_t)\exp(fg). 
\end{equation}
\end{itemize}

$(3)$ Parabolic coordinates: $\xi =\frac{1}{2}(u^2-v^2)$, $\eta =uv$. The Euclidean metric $g=d\xi ^2+d\eta ^2$ becomes a Liouville one: $g=(u^2+v^2)(du^2+dv^2)$ and the Ricci flow equation
is now {\it the parabolic Ricci flow equation}:
\begin{equation}
\label{3.9}
2\sqrt{\xi ^2+\eta ^2}(h_{\xi \xi }+h_{\eta \eta })=(\exp h)_t.
\end{equation}
\begin{itemize}
\item[(3.1)] With $h=\varphi (t, w=\xi +\alpha \eta )$, we derive {\it the parabolic solitonic Ricci flow equation}:
\begin{equation}
\label{3.10}
2\sqrt{\xi ^2+\eta ^2}(1+\alpha ^2)\varphi _{ww}=\varphi _t(\exp \varphi ). 
\end{equation}
\item[(3.2)] For $h=f(t, \xi )g(t, \eta )$ we get:
\begin{equation}
\label{3.11}
2\sqrt{\xi ^2+\eta ^2}(f_{\xi \xi }g+fg_{\eta \eta })=(f_tg+fg_t)(\exp fg). 
\end{equation}
\item[(3.3)] For $h=f(t, \xi )+g(t, \eta )$ we obtain:
\begin{equation}
\label{3.12}
2\sqrt{\xi ^2+\eta ^2}(f_{\xi \xi }+g_{\eta \eta })=(f_t+g_t)\exp (f+g). 
\end{equation}
\end{itemize}

$(4)$ Elliptic coordinates: $x^2=c^2(u-1)(v-1)$, $y^2=-c^2uv$. The Euclidean metric $g=dx^2+dy^2$ becomes a Lorentzian one:
$g=\frac{c^2(v-u)}{4}\left(\frac{du^2}{u(u-1)}-\frac{dv^2}{v(v-1)}\right)$ and the Ricci flow equation is now {\it the elliptic Ricci flow equation}:
\begin{align}
\label{3.13}
(\exp h)_t=&h_x\Delta _{u, v}x+h_y\Delta _{u, v}y\\ \notag
+&\frac{c^2}{4}\left[h_{xx}\frac{v-u}{u(u-1)}+2h_{xy}\left(\frac{v(1-v)}{u(u-1)}+\frac{u(1-u)}{v(v-1)}\right)+h_{yy}\frac{u-v}{v(v-1)}\right],
\end{align}
where $\Delta _{u, v}$ is the Euclidean Laplacian in the variables $(u, v)$; supposing $c>0$ we have:
\begin{equation}
\label{3.14}
\left\{ \begin{array}{ll}
\Delta _{u, v}x=-\frac{c}{4}\left(\frac{1}{u-1}\sqrt{\frac{v-1}{u-1}}+\frac{1}{v-1}\sqrt{\frac{u-1}{v-1}}\right), \\
\\
\Delta _{u,
v}y=-\frac{c}{4}\left(\frac{1}{u}\sqrt{\frac{-v}{u}}+\frac{1}{v}\sqrt{\frac{-u}{v}}\right). 
\end{array} \right.
\end{equation}
\begin{itemize}
\item[(4.1)] With $h=\varphi (t, w=x+\alpha y)$ we obtain {\it the elliptic solitonic Ricci flow equation}:
\begin{align}
\label{3.15}
\varphi _t\exp \varphi=&\varphi _w(\Delta _{u, v}x+\alpha h_y\Delta _{u, v}y)\\ \notag
+&\varphi _{ww}[\frac{v-u}{u(u-1)}+2\alpha \left(\frac{v(1-v)}{u(u-1)}+\frac{u(1-u)}{v(v-1)}\right)+\alpha
^2\frac{u-v}{v(v-1)}]. 
\end{align}
\item[(4.2)] With $h=f(t, x)g(t, y)$ we get:
\begin{align}
\label{3.16}
(f_tg+fg_t)(\exp fg)=&f_xg\Delta _{u, v}x+fg_y\Delta _{u, v}y+f_{xx}g\frac{v-u}{u(u-1)}\\ \notag
+&2f_{x}g_{y}\left(\frac{v(1-v)}{u(u-1)}+\frac{u(1-u)}{v(v-1)}\right)+fg_{yy}\frac{u-v}{v(v-1)}.
\end{align}
\item[(4.3)] With $h=f(t, x)+g(t, y)$ we obtain:
\begin{align}
\label{3.17}
(f_t+g_t)\exp (f+g)=&f_x\Delta _{u, v}x+g_y\Delta _{u, v}y\\ \notag
+&f_{xx}\frac{v-u}{u(u-1)}+g_{yy}\frac{u-v}{v(v-1)}. 
\end{align}
\end{itemize}
On this way, we add to the PDE flavor of the topic of Ricci flow three new equations: the polar, the parabolic and the elliptic Ricci flow equation, as well as their solitonic
companions.

\section*{Acknowledgement} 

The second author expresses her appreciation to The Scientific and Technological Research Council of Turkey (T\"UB{\.I}TAK) for  financial support (Grant Number:1059B141600696) during her research at the Mathematical Department of "Gh. Asachi" Technical University of Iasi. 


\end{document}